\documentclass[12pt,reqno]{amsart}
\usepackage{amssymb}
\usepackage{graphicx}
\usepackage{xcolor}

\usepackage[all]{xy}
\DeclareFontFamily{U}{mathb}{\hyphenchar\font45}
\DeclareFontShape{U}{mathb}{m}{n}{
      <5> <6> <7> <8> <9> <10> gen * mathb
      <10.95> mathb10 <12> <14.4> <17.28> <20.74> <24.88> mathb12
      }{}
\DeclareSymbolFont{mathb}{U}{mathb}{m}{n}
\DeclareMathSymbol{\righttoleftarrow}{3}{mathb}{"FD}

\theoremstyle{plain}
\newtheorem{prop}{Proposition}[section]
\newtheorem{theo}[prop]{Theorem}

\newtheorem{lemm}[prop]{Lemma}
\theoremstyle{remark}
\newtheorem{rema}[prop]{Remark}

\theoremstyle{definition}

\newtheorem{exam}[prop]{Example}
\numberwithin{equation}{section}

\newcommand{\PP}{{\mathbb P}}
\newcommand{\bP}{{\mathbb P}}
\newcommand{\Q}{{\mathbb Q}}

\newcommand{\G}{{\mathbb G}}

\newcommand{\R}{{\mathbb R}}
\newcommand{\Z}{{\mathbb Z}}

\newcommand{\cB}{{\mathcal B}}

\newcommand{\cD}{{\mathcal D}}

\newcommand{\cG}{{\mathcal G}}
\newcommand{\cO}{{\mathcal O}}

\newcommand{\cM}{{\mathcal M}}

\newcommand{\cX}{{\mathcal X}}

\newcommand{\fK}{{\mathfrak K}}

\newcommand{\rH}{{\mathrm H}}

\newcommand{\rN}{{\mathrm N}}

\newcommand{\bG}{{\mathbb G}}

\newcommand{\bQ}{{\mathbb Q}}

\newcommand{\bZ}{{\mathbb Z}}

\newcommand{\fS}{{\mathfrak S}}

\newcommand{\eqto}{\stackrel{\lower1.5pt\hbox{$\scriptstyle\sim\,$}}\to}
\newcommand{\eqdashto}{\stackrel{\lower1.5pt\hbox{$\scriptstyle\sim\,$}}\dashrightarrow}

\DeclareMathOperator{\Am}{Am}

\DeclareMathOperator{\DVal}{\mathcal D Val}
\DeclareMathOperator{\Gal}{Gal}
\DeclareMathOperator{\GL}{GL}
\DeclareMathOperator{\SL}{SL}
\DeclareMathOperator{\PGL}{PGL}

\DeclareMathOperator{\Pic}{Pic}
\DeclareMathOperator{\Spec}{Spec}

\DeclareMathOperator{\Hom}{Hom}

\DeclareMathOperator{\Br}{Br}
\DeclareMathOperator{\Aut}{Aut}
\DeclareMathOperator{\Ker}{Ker}

\begin{document}
\title[Brauer group of quotients]{Unramified Brauer group of quotient spaces by finite groups}

\author{Andrew Kresch}
\address{
  Institut f\"ur Mathematik,
  Universit\"at Z\"urich,
  Winterthurerstrasse 190,
  CH-8057 Z\"urich, Switzerland
}
\author{Yuri Tschinkel}
\address{
  Courant Institute,
  251 Mercer Street,
  New York, NY 10012, USA
}

\email{tschinkel@cims.nyu.edu}

\address{Simons Foundation\\
160 Fifth Avenue\\
New York, NY 10010\\
USA}

\date{January 16, 2024}

\begin{abstract}
We provide a general algorithm for the computation of the unramified Brauer group of quotients of rational varieties by finite groups.  
\end{abstract}

\maketitle

\section{Introduction}
\label{sec.intro}
Let $V$ be a variety over an algebraically closed field $k$ of characteristic zero and $G$ a finite group acting generically freely on $V$. For example, $V$ could be a 
finite-dimensional faithful representation of $G$. The rationality problem for the field of invariants 
$$
K=k(V)^G=k(V/G)
$$ 
has attracted the attention of many mathematicians, e.g., in connection with Noether's problem (see \cite{CTS} for a survey and further references). 

One of the obstructions is the {\em unramified Brauer group} 
$$
\Br_{\mathrm{nr}}(K)\cong \Br(X)=\rH^2(X,\bG_m),
$$
which coincides with the Brauer group of a smooth projective model $X$ of $K$. 
By a result of Bogomolov \cite{Bog-linear} (see  also \cite[Thm.\ 6.1]{CTS}), this group can be computed in terms of the set 
$\cB_G$  
of \emph{bicyclic} subgroups of $G$:
\begin{equation}
\label{eqn.bicyclicnr}
\Br_{\mathrm{nr}}(k(V)^G) =\{\alpha \in \Br(k(V)^G) \, | \, \alpha_A \in \Br_{\mathrm{nr}}(k(V)^A), \, \forall A\in \cB_G\}.
\end{equation}
This yields explicit formulas in special cases.  
\begin{itemize}
\item[(1)] If $V$ is a faithful representation of $G$ then (cf.\ \cite[Thm.\ 7.1]{CTS})
$$
\Br_{\mathrm{nr}}(K)\cong \ker\Big(\rH^2(G,\bQ/\bZ)  \to \bigoplus_{A\in \mathcal B_G} \rH^2(A,\bQ/\bZ)\Big).
$$
\item[(2)] 
If $V=T$ is an algebraic torus over $k$, 
with $G$-action arising from an injective homomorphism $G\to \Aut(M)$, where $M=\mathfrak{X}^*(T)$,
then (cf.\ \cite[Thm.\ 8.7]{CTS})
$$\Br_{\mathrm{nr}}(K)\cong \ker\Big(\rH^2(G,\Q/\Z \oplus M) \to \bigoplus_{A\in \mathcal B_G} \rH^2(A, \Q/\Z \oplus M)\Big). 
$$
\item[(3)] The case
$V=\mathrm{SL}_n$ with $G\subset \mathrm{SL}_n$ acting by translations, is treated in \cite{CT-SL} and, by means of a stable equivariant birational equivalence to a linear action, leads to
the same outcome as case (1). 
\end{itemize}

After some preliminary material (Sections \ref{sect:gen} and \ref{sect:Bomu}), we highlight the role of the Brauer group of the quotient \emph{stack} 
$$[V/G]$$
(Section \ref{sect:stack})
and give a uniform treatment of some known (Section \ref{sect:basic}) and new cases ($V$ a projective space in Section \ref{sect:basic}, a Grassmannian variety in Section \ref{sect:grass}, a flag variety in Section \ref{sect:flag}). 
The main result (Section \ref{sect:destack}) is a general procedure to determine the unramified Brauer group
$\Br_{\mathrm{nr}}(k(V)^G)$ 
for a $G$-action on a rational variety $V$.

%We use this to provide new examples of obstructions to $G$-equivariant birationality of actions of finite groups. 

\medskip
\noindent
{\bf Acknowledgments:} 
We are grateful to Fedor Bogomolov for his interest and comments. 
The second author was partially supported by NSF grant 2301983.

\section{Generalities}
\label{sect:gen}
We work over an algebraically closed field $k$ of characteristic zero.

\subsection*{Group cohomology}
As recalled in \cite[\S 2.1]{KT-dp}, there is a natural identification
\[ \rH^i(G,k^\times)\cong \rH^i(G,\mu_\infty)\qquad (i\ge 1) \]
of group cohomology for any finite group $G$
with trivial action on $k^\times$, respectively $\mu_\infty$.
%This can be expressed in stack language as $\rH^i(BG,\G_m)$.
We identify $\mu_\infty$ with $\Q/\Z$ and write
$$
\rH^i(G)=\rH^i(G,\bQ/\bZ).
$$
For $i=1$ we have $\rH^1(G):=\Hom(G,\Q/\Z)$, and for $i=2$,
an interpretation of $\rH^2(G)$ in terms of central extensions of $G$; see \cite[\S IV.3]{brown}.

For any subgroup $A\subseteq G$  we denote by
$$
\mathrm{res}^i_A\colon \rH^i(G)\to \rH^i(A)
$$
the restriction homomorphism.
For a normal subgroup with $Q=G/A$,
the Hochschild-Serre spectal sequence yields the long exact sequence
\[
0\to \rH^1(Q)\to \rH^1(G)\to \rH^1(A)^Q\to \rH^2(Q)\to \ker(\mathrm{res}^2_A)\to \rH^1(Q,\rH^1(A)).
\]
This gives two split short exact sequences when $G=A\rtimes Q$.

For $G$ cyclic with generator $g$ and a $G$-module $M$ the group cohomology $\rH^i(G,M)$ can be identified with the cohomology of the complex
\[ M \stackrel{\Delta}\to M \stackrel{N}\to M \stackrel{\Delta}\to M \dots, \]
where $\Delta=g-1$ and $N=1+g+\dots+g^{n-1}$ ($n=|G|$),
cf.\ \cite[Exa.\ III.1.2]{brown}.
The case $G$ is abelian, expressed as a product of cyclic groups, may be treated via tensor product of resolutions corresponding to the factors as described in \cite[Prop.\ V.1.1]{brown}, e.g.,
for bicyclic $G\cong G_1\times G_2$ with correspnding $\Delta_i$ and $N_i$, $i=1$, $2$:
\[
M \stackrel{\big(\arraycolsep=1pt\begin{array}{cc}\scriptscriptstyle\Delta_1\\[-2pt]\scriptscriptstyle\Delta_2\end{array}\big)}\longrightarrow
M^2 \stackrel{\bigg(\arraycolsep=1pt\begin{array}{cc}\scriptscriptstyle N_1& \scriptscriptstyle 0\\[-2pt]\scriptscriptstyle -\Delta_2 & \scriptscriptstyle \Delta_1\\[-2pt]\scriptscriptstyle 0 & \scriptscriptstyle N_2\end{array}\bigg)}\longrightarrow
M^3 \dots.
\]
We see easily, this way, that $\rH^2(G)=0$ when $G$ is cyclic, and
\[ \rH^2(G_1\times G_2)\cong \Z/d\Z,\qquad d=\gcd(n_1,n_2), \]
for cyclic $G_i$ of order $n_i$ for $i=1$, $2$ (cf.\ \cite[\S 2.1]{KT-dp}).

\subsection*{Fields} Throughout, $K=k(V)$ is the function field of an algebraic variety $V$ over
$k$.
We write $\DVal_K$ for the set of 
divisorial valuations of $K$. Every $\nu\in\DVal_K$ can be realized as a valuation corresponding to a divisor on some smooth projective model of $K$. 
%We write $\mathfrak c_{\nu}(X)$ for the center of $\nu$ on a given model $X$ of $K$, and $K_{\nu}$ for the completion of $K$ at $\nu$. 

\subsection*{Unramified cohomology}
Let $\nu\in \DVal_K$ with residue field $\kappa$ and absolute Galois group
$\cG_\kappa$ of $\kappa$.
There is a residue homomorphism
$$
\partial_{\nu}\colon \Br(K) \to \rH^1_{\mathrm{cont}}(\cG_{\kappa})=\Hom_{\mathrm{cont}}(\cG_{\kappa},\Q/\Z) 
$$
with values in the continuous group cohomology. We have
$$
\Br_{\mathrm{nr}}(K)\subset \Br(K),\qquad \Br_{\mathrm{nr}}(K)=\bigcap_{\nu\in \DVal_K} \Ker(\partial_{\nu}),
$$
with
$\Br_{\mathrm{nr}}(K)\cong \Br(X)$ for any smooth projective model $X$ of $K$.
The group $\Br_{\mathrm{nr}}$ is invariant under purely transcendental extensions.
In particular, a rational variety $V$ has $\Br_{\mathrm{nr}}(k(V))=0$.

An important result, Fischer's theorem \cite{fischer}, asserts the rationality of $V/A$ for a linear action of an abelian group $A$.
Then $\Br_{\mathrm{nr}}(k(V)^A)=0$.

%\subsection*{Standard form}
%Let $V$ be a smooth projective variety with faithful regular $G$-action. 
%We say that the $G$-action is in \emph{standard form} if there exists a Zariski open $U\subset V$ where the action is free, with simple normal crossing boundary 
%$$
%V\setminus U=D=\bigcup_{j \in J} D_j,
%$$
%where the $D_j$ are irreducible, such that
%for all $g\in G$ and $j\in J$ we have $g(D_j)=D_j$ or $g(D_j)\cap D_j=\emptyset$. 
%Such a form can be reached after equivariant blowups. 
%On $V$ in standard form, all stabilizers are abelian. 

\subsection*{Basic exact sequence}

Let $V$ be a smooth projective $G$-variety over $k$.
Assume that $V$ is rational.
The Leray spectral sequence, applied to the morphism from the Deligne-Mumford stack (DM stack) $[V/G]$, associated with the $G$-action on $V$, to
the stack $BG$ of $G$-torsors, yields the long exact sequence
\begin{align}
\begin{split}
\label{eqn.BrXmodG}
&0\to 
\Hom(G,k^\times)\to \Pic(V,G)\to 
\Pic(V)^G \stackrel{\delta_2}\to  \rH^2(G,k^\times)\\
&\to \Br([V/G])\to \rH^1(G,\Pic(V))\stackrel{\delta_3}\to  \rH^3(G,k^\times)\to \rH^3([V/G],\G_m),
\end{split}
\end{align}
where $\Pic(V,G)$ denotes the group of isomorphism classes of $G$-linearized line bundles.
In \cite{KT-dp} this is used to exhibit $G$-actions on rational surfaces with obstructions to (stable) linearizability of the $G$-action, e.g., 
nonvanishing of 
\begin{itemize}
\item 
the Amitsur group 
$
\mathrm{Am}(V,G):=\mathrm{im}(\delta_2)
$
(see \cite[Sect.\ 6]{blanc2018finite}), 
\item the image $\mathrm{im}(\delta_3)$, 
\item 
the cohomology $\rH^1(G,\Pic(V))$. 
\end{itemize}

If $V$ has a $G$-fixed point, then by basic functoriality the map from $\rH^2(G,k^\times)=\Br(BG)$ to $\Br([V/G])$ is injective, thus $\delta_2=0$, and
similarly, $\delta_3=0$.

If $V$ is quasiprojective then the Leray spectral sequence leads to a basic exact sequence
with first term $\rH^1(G,\G_m(V))$ and
$\rH^i(G,k^\times)$ ($i=2$, $3$) replaced by $\rH^i(G,\G_m(V))$
and $\Br([V/G])$ by $\ker(\Br([V/G])\to \Br(V))$.

We will use the following observation, which appears in \cite{villalobospaz}.

\begin{lemm}
\label{lem.PicVW}
Suppose $V\to W$ is a $G$-equivariant morphism of smooth projective $G$-varieties, such that the induced homomorphism
\[ \Pic(W)\to \Pic(V) \]
is injective (resp., an isomorphism).
Then $\Pic(W,G)\to \Pic(V,G)$ is injective (resp., an isomorphism), and
$\mathrm{Am}(W,G)$ is contained in (resp., is equal to) $\mathrm{Am}(V,G)$.
\end{lemm}

\begin{proof}
We have the commutative diagram
\[
\xymatrix{
0\ar[r] & \Hom(G,k^\times)\ar[r]\ar@{=}[d] & \Pic(W,G)\ar[r]\ar[d] & \Pic(W)^G \ar[r]^{\delta_2}\ar[d] & \rH^2(G,k^\times)\ar@{=}[d]
\\
0\ar[r] & \Hom(G,k^\times)\ar[r] & \Pic(V,G)\ar[r] & \Pic(V)^G\ar[r]^{\delta_2} & \rH^2(G,k^\times)
}
\]
with exact rows.
The result follows.
\end{proof}

\subsection*{Linearized bundles}
Let $V$ be a smooth projective $G$-variety over $k$ and $E$ a vector bundle over $V$.
We suppose
that the projectivization $\PP(E)$ is endowed with a $G$-action, so that the projection to $V$ is $G$-equivariant, and we have a central cyclic extension
\begin{equation}
\label{eqn.ZGextn}
1\to Z\to \widetilde{G}\to G\to 1
\end{equation}
and a compatible $\widetilde{G}$-linearization of $E$,
with scalar action of $Z$.
We may suppose the latter, by replacing $Z$ and $\widetilde{G}$ by suitable quotients, to be by the identity character of $Z=\mu_\ell$, $\ell=|Z|$.
Then:
\begin{itemize}
\item A splitting of \eqref{eqn.ZGextn} leads to a $G$-linearization of $E$.
\item Generally, \eqref{eqn.ZGextn} determines a class $\gamma_E\in \rH^2(G)$, obstruction to existence of a splitting (for sufficiently divisible $\ell$).
\item We have $\gamma_{E\otimes E'}=\gamma_E+\gamma_{E'}$.
\item A line bundle $L$ with $[L]\in \Pic(V)^G$ leads to $\gamma_L=\delta_2([L])$.
\end{itemize}
%We suppose that $V$ is rational.
%We wish to investigate the unramified Brauer groups of the invariants of the function fields of $V$, $\PP(E)$, and $E$.
%not sure what sequence to write down, basic exact sequence for $[V/G]$ and $[\PP(E)/G]$ possible, also Leray for $[\PP(E)/G]\to [V/G]$ possible
%The basic exact sequence \eqref{eqn.BrXmodG} yields:
%\[
%\xymatrix{
%0 \ar[r] &\mathrm{Am}(V)\ar[d] \ar[r] & \rH^2(G) \ar@{=}[d] \ar[r] & \Br([V/G])\ar[d] \ar[r] & \rH^1(G,V) \ar@{=}[d]   \\
%0 \ar[r] &\mathrm{Am}(\bP(E)) \ar[r] & \rH^2(G) \ar[r] &  \Br([\bP(E)/G])\ar[r] &  \rH^1(G,\bP(E)) 
%}
%\]
%certainly want to record:
%
%case of a line bundle $E=L$, with invariant class in $\Pic(V)$, get obstruction $\delta_2([L])\in \Am(V,G)$
%to linearizing $L$
%
%case of a higher rank vector bundle $E$, get class $\gamma\in \Br([V/G])$, which vanishes if and only if there exists a line bundle $L$, so that $E\otimes L$ admits a $G$-linearization which is compatible with the given $G$-action on $\PP(E\otimes L)\cong \PP(E)$ (then how uniquely is $L$ determined?)
If the $G$-action on $V$ is generically free and $E$ admits a $G$-linearization, then $k(E)^G$ is a purely transcendental extension of $k(V)^G$; this is known as the No-Name Lemma, see \cite[Sect.\ 4.3]{ChGR}.

\begin{exam}
\label{exa.linear}
Let $V^\circ$ be a $k$-vector space of dimension $n$ with projectivization $V=\PP(V^\circ)$, and let $G$ act on $V$.
We adopt the convention that this is a right action.
So it is given by a homomorphism
$G\to \PGL(V^{\circ\vee})$.
We have, canonically, a central cyclic extension \eqref{eqn.ZGextn} and compatible $\widetilde{G}\to \SL(V^{\circ\vee})$, with $Z=\mu_n$.
Then \eqref{eqn.ZGextn} determines an $n$-torsion class
\[ \gamma=\delta_2([\cO_V(-1)])\in \rH^2(G), \]
with
\[ \Am(V,G)=\langle \gamma\rangle. \]
For the trivial bundle
$\underline{V}^\circ$ associated with the given vector space we have the given $G$-action on the projectivization and as above a $\widetilde{G}$-linearization, thus
$\gamma_{\underline{V}^\circ}=\gamma$.
The corresponding $\widetilde{G}$-linearization of $E=\underline{V}^\circ\otimes \cO_V(1)$ has trivial $Z$-character, and we get a canonical $G$-linearization of $E$.
\end{exam}

\section{Bogomolov multiplier}
\label{sect:Bomu}
The description of $\Br_{\mathrm{nr}}(k(V)^G)$ for a faithful representation of $G$ from special case (1) of the Introduction involves a subgroup of $\rH^2(G)$, known as the \emph{Bogomolov multiplier}:
\[ \mathrm{B}_0(G):=\ker\Big(\rH^2(G,\bQ/\bZ)  \to \bigoplus_{A\in \mathcal B_G} \rH^2(A,\bQ/\bZ)\Big). \]
Here, $\mathcal{B}_G$ denotes the set of bicyclic subgroups of $G$.
In this section we recall some facts about
$\mathrm{B}_0(G)$, including its vanishing for some classes of groups $G$.
All groups $G$, $A$, etc., considered in this section, are finite.

The following facts follow from the long exact sequence coming from the Hochschild-Serre spectral sequence, recalled in Section \ref{sect:gen}:
\begin{itemize}
\item If $G\to A$ is a surjective homomorphism of abelian groups, then the induced homomorphism
$\rH^2(A)\to \rH^2(G)$ is injective.
\item If $G$ is abelian, $G=G_1\times\dots\times G_r$ with cyclic factors $G_i$, then
\[ \rH^2(G)\cong \bigoplus_{i<j} \rH^2(G_i\times G_j). \]
\end{itemize}
By the second fact, the Bogomolov multiplier of a group $G$ may be defined equivalently with direct sum over all abelian subgroups $A$ of $G$ (as in \cite{Bog-linear}).

\begin{lemm}
\label{lem.abexcy}
Assume that there is a short exact sequence of groups
\[ 1\to A\to G\to C\to 1, \]
where $A$ is abelian and $C=\langle c\rangle$ is cyclic,
and let $0\ne \alpha\in \rH^2(G)$ be given, with $\mathrm{res}_A^2(\alpha)=0$.
Then there exists an element $a\in A$, in the center of $G$, such that for any lift $b\in G$ of $c$ we have
$\mathrm{res}_{\langle a,b\rangle}^2(\alpha)\ne 0$.
In particular, $\mathrm{B}_0(G)=0$.
\end{lemm}

Proofs of this and similar statements make use of the long exact sequence coming from the Hochschild-Serre spectral sequence and the descriptions of group cohomology of abelian groups, given in Section \ref{sect:gen}.

\begin{proof}
The class $\alpha\in \ker(\mathrm{res}_A^2)$ determines a class $0\ne \tilde\alpha\in \rH^1(C,A^\vee)$,
where $A^\vee$ denotes $\Hom(A,\Q/\Z)$.
We employ the notation $\Delta$ and $N$ for $A$ as $C$-module, and equally well for $A^\vee$.
Under the identification of $\rH^1(C,A^\vee)\cong\ker(N)/\Delta(A^\vee)$, a representative $\tilde\chi\in A^\vee$, $N(\chi)=0$, may be chosen so that $\ker(\tilde\chi)$ contains $\Delta^i(A)$ (the image of the $i$th iterate of $\Delta$) for some positive integer $i$.
We suppose this is done, with $i$ as small as possible.
Then $\tilde{\chi}|_{\Delta^{i-1}(A)}$ does not lie in the image of the map
\[
(\Delta^i(A)/\Delta^{i+1}(A))^\vee\to (\Delta^{i-1}(A)/\Delta^i(A))^\vee
\]
induced by $\Delta$.
(The existence of $\chi\in A^\vee$ with $\Delta^{i+1}(A)\subset \ker(\chi)$
and $\Delta(\chi)|_{\Delta^{i-1}(A)}=\tilde{\chi}|_{\Delta^{i-1}(A)}$ would contradict the minimality of $i$.)
Consequently, there exists
\[ \bar a\in \ker\big(\Delta^{i-1}(A)/\Delta^i(A)\to \Delta^i(A)/\Delta^{i+1}(A)\big),\qquad \bar a\notin \ker(\tilde\chi). \]
There is then a lift $a\in \Delta^{i-1}(A)$, belonging to the center of $G$,
and this satisfies the desired property.
\end{proof}

The conclusion
$\mathrm{B}_0(G)=0$ is known \cite[Lemma 4.9]{Bog-linear}.
We use the description of the indicated bicyclic subgroups of $G$ in Lemma \ref{lem.abexcy} to give a direct proof of the next lemma, established using different methods (group homology of certain universal semidirect products) in \cite{Barge}.

\begin{lemm}
\label{lem.AsemiB}
Suppose that $G=A\rtimes B$ is a semidirect product of abelian groups $A$ and $B$, with $B$ bicyclic.
Then $\mathrm{B}_0(G)=0$.
\end{lemm}

\begin{proof}
Suppose $0\ne \alpha\in \rH^2(G)$ with $\mathrm{res}_A^2(\alpha)=0=\mathrm{res}_B^2(\alpha)$.
Then the class $\tilde\alpha\in \rH^1(B,A^\vee)$, determined by $\alpha$, is nonzero.

We represent $B$ as a product of a pair of cyclic subgroups and employ corresponding notation $\Delta_1$, $N_1$, $\Delta_2$, $N_2$.
Then $\tilde\alpha$ may be represented by
\[ (\tilde\chi,\tilde\chi')\in A^\vee\times A^\vee, \]
satisfying $N_1(\tilde\chi)=0=N_2(\tilde\chi')$ and $\Delta_2(\tilde\chi)=\Delta_1(\tilde\chi')$.
This is unique up to coboundaries of the form
$(\Delta_1(\chi),\Delta_2(\chi))$ for $\chi\in A^\vee$.

The product representation $B=C_1\times C_2$ determines subgroups $G_i=A\rtimes C_i$ ($i=1$, $2$) of $G$.
If $\mathrm{res}_{G_2}^2(\alpha)\ne 0$, then Lemma \ref{lem.abexcy} supplies a bicyclic subgroup $\langle a,b\rangle$ of $G_2$ with $\mathrm{res}_{\langle a,b\rangle}^2(\alpha)\ne 0$, so we suppose, instead,
$\mathrm{res}_{G_2}^2(\alpha)=0$.
Then $\tilde\chi'=\Delta_2(\chi')$, for some $\chi'\in A^\vee$, and, modifying the cocycle representative by a coboundary, we are reduced to the case
\[ \tilde\chi'=0. \]
So $\Delta_2(\tilde\chi)=0$, i.e., $\tilde\chi\in (A/\Delta_2(A))^\vee$,
and $\tilde\chi$ determines
\[ \beta\in \ker\big(
\rH^2(A/\Delta_2(A)\rtimes C_1)\to \rH^2(A/\Delta_2(A)\big), \]
mapping to $\alpha\in \rH^2(G)$.

We apply Lemma \ref{lem.abexcy} to $\beta$ to obtain $\bar a\in A/\Delta_2(A)$ in the center of $A/\Delta_2(A)\rtimes C_1$ and a set $\mathcal{B}_{\bar a}$ of bicyclic subgroups, to which $\beta$ restricts nontrivially.
Let $a$ be a lift to $A$.
Then
$\Delta_1(a)=\Delta_2(b)$ for some $b\in A$.
Now the elements of $G$, obtained by pairing $a$ with chosen generator of $C_2$, and $b$ with chosen generator of $C_1$, generate an abelian subgroup of $G$ whose image in $A/\Delta_2(A)\rtimes C_1$ is in $\mathcal{B}_{\bar a}$.
This concludes the proof.
\end{proof}

\begin{lemm}
\label{lem.centralofbicyclic}
Suppose that $G$ is a central extension of a bicyclic group.
Then $\mathrm{B}_0(G)=0$.
\end{lemm}

\begin{proof}
We write a central exact sequence of groups
\[ 1\to A\to G\to B\to 1, \]
with $B$ bicyclic.
The proof will use the easy observation that $G$ is abelian if and only if $\rH^1(G)$ maps surjectively to $A^\vee$ (cf.\ the long exact sequence coming from the Hochschild-Serre spectral sequence).

Let a given $0\ne \alpha\in \rH^2(G)$, with $\mathrm{res}^2_A(\alpha)=0$, determine a class $\tilde\alpha\in \rH^1(B,A^\vee)=\Hom(B,A^\vee)$.
If $\tilde\alpha\ne 0$, then $\alpha$ remains nonzero upon restriction to the pre-image in $G$ of a suitable cyclic subgroup of $B$, and we may conclude by Lemma \ref{lem.abexcy}.
We suppose $\tilde \alpha=0$, thus $\alpha\in \rH^2(G)$ is the image under 
\[ \rH^2(B)\to \rH^2(G) \]
of some $\alpha_0\in \rH^2(B)$.
We write $B=C_1\times C_2$, cyclic subgroups of orders $|C_1|=n_1$ and $|c_2|=n_2$, so
$\rH^2(B)\cong \Z/d\Z$ with $d=\gcd(n_1,n_2)$.

Let $e$ denote the order of the image of $A^\vee\to \rH^2(B)$ (the transgression map, coming from the Hochschild-Serre spectral sequence) and $f$ the order of $\alpha_0\in \rH^2(B)$.
We have $f\nmid e$, since $\alpha\ne 0$.
Restriction from $B$ to the subgroup $eB$ leads to the class $0\ne \bar \alpha_0\in \rH^2(eB)\cong \Z/(d/e)\Z$.
Letting $G'$ denote the pre-image of $eB$ in $G$, the corresponding Hochschild-Serre spectral sequence gives a trivial transgression map, hence surjective $\rH^1(G')\to A^\vee$.
Therefore $G'$ is abelian, and $\mathrm{res}^2_{G'}(\alpha)\ne 0$.
\end{proof}

\begin{rema}
\label{rem.sharpness}
Lemmas \ref{lem.abexcy} through \ref{lem.centralofbicyclic} are somewhat sharp.
There exist groups $G$, extensions by abelian groups of bicyclic groups with $\mathrm{B}_0(G)\ne 0$;
an example is given in \cite[Sect.\ 4]{Bog-linear}.
For $p$ prime, \cite[Sect.\ 5]{Bog-linear} investigates and exhibits $p$-groups $G$ with $[G,[G,G]]=0$ and $\mathrm{B}_0(G)\ne 0$;
subject to a minimality condition it is shown that $G/[G,G]\cong (\Z/p\Z)^{2m}$, $m\ge 2$.
\end{rema}

\section{Brauer group of the quotient stack}
\label{sect:stack}

In \cite{KT-dp}, we explained the computation of $\Br([V/G])$ in case $V$ is a rational surface. 
Now, $V$ is a smooth projective rational variety of arbitrary dimension, and we give a description of $\Br([V/G])$ as a subgroup of
\begin{equation}
\label{eqn.H2GkV}
\rH^2(G,k(V)^\times)\cong \ker\big(\Br(k(V)^G)\to \Br(k(V))\big).
\end{equation}

We refer to the basic exact sequence of Section \ref{sect:gen}.
A subgroup, isomorphic to $\rH^2(G,k^\times)/\mathrm{Am}(V,G)$, gives rise directly, via $k^\times\hookrightarrow k(V)^\times$, to elements of $\rH^2(G,k(V)^\times)$.
To complete the description, we need to explain how to lift elements of $\ker(\delta_3)$ to the group \eqref{eqn.H2GkV}.
For this, we
take a $G$-invariant collection of divisors $D_i$, generating $\Pic(V)$,
introduce the exact sequences of $G$-modules
\[ 0\to R\to \bigoplus_{i} \Z\cdot [D_i]\to \Pic(V)\to 0 \]
and, with complement $U$ in $V$ of $D=\bigcup_{i}D_i$ and corresponding exact sequence
\[ 0\to k^\times\to \G_m(U)\to R\to 0 \]
of $G$-modules,
consider the diagram (see \cite[Sect.\ 6]{KT-effect}):
\[
\xymatrix@C=12pt@R=15pt{
&&
\rH^2(G,\bG_m(U))\ar[d] & \\
0 \ar[r] &
\rH^1(G,\Pic(V)) \ar[r]\ar[dr]_{\delta_3} &
\rH^2(G,R) \ar[r] \ar[d] &
\rH^2(G,\bigoplus_{i} \Z\cdot[D_i]) \\
&&\rH^3(G,k^\times)&
}
\]
Given an element of $\ker(\delta_3)$, its image in $\rH^2(G,R)$ may be lifted to $\rH^2(G,\bG_m(U))$.
We obtain a representative in $\rH^2(G,k(V)^\times)$ of a corresponding Brauer class on $[V/G]$.

We also recall the formulation of purity.
Here, $V$ need not be projective or rational, but we suppose that $G$ acts generically freely on $V$.
An element $\alpha\in \Br(k(V)^G)$ comes from $\Br([V/G])$ if and only if it has vanishing residue along the divisors of $[V/G]$ \cite[Prop.\ 2.2]{bssurf}.
The residues along divisors of $[V/G]$ are related to the classical residues (Section \ref{sect:gen}) as follows.
We fix an irreducible divisor on $[V/G]$, corresponding to a $G$-orbit $D=D_1\cup\dots\cup D_m$ of components on $V$, and suppose that each $D_i$ has generic stabilizer of order $n$.
Then \cite[Lemma 4.1]{KT-dp} the residue of $\alpha$ along the divisor $[D/G]$ of $[V/G]$ is equal to $n\delta_\nu(\alpha)$, where $\nu\in \DVal_{k(V)^G}$ is the associated divisorial valuation of the function field $k(V)^G$ of $V/G$.

For $G$ acting generically freely on smooth projective rational $V$ we have inclusions
\[ \Br_{\mathrm{nr}}(k(V)^G)\subset \Br([V/G])\subset \Br(k(V)^G). \]
Indeed, the defining conditions for $\Br_{\mathrm{nr}}(k(V)^G)$ are vanishing $\delta_\nu$ for all $\nu\in \DVal_{k(V)^G}$, while for the purity characterization of $\Br([V/G])$ only the $\nu$ associated with divisors on $[V/G]$ are involved, and then only the vanishing of $n_\nu\delta_\nu$ is required, for some positive integer $n_\nu$.
Since $\Br([V/G])$ is contained in the kernel of $\Br(k(V)^G)\to \Br(k(V))$, using \eqref{eqn.H2GkV} we have
\begin{equation}
\label{eqn.containments}
\Br_{\mathrm{nr}}(k(V)^G)\subset \Br([V/G])\subset \rH^2(G,k(V)^\times). \end{equation}

\begin{lemm}
\label{lem.fixedpoint}
Let $A$ be an abelian group, acting generically freely on a smooth projective variety $V$,
and let $\alpha\in \Br([V/A])$.
For $v\in V^A$ we denote by $$
i_v^*\colon \Br([V/A])\to \rH^2(A,k^\times)
$$ 
the corresponding splitting in the basic exact sequence.
If $\alpha\in \Br_{\mathrm{nr}}(k(V)^A)$,
then $i_v^*(\alpha)=0$, for all $v\in V^A$.
\end{lemm}

\begin{proof}
Replacing $V$ by $V\times \PP^1$ if needed (with trivial $A$-action on $\PP^1$), we may suppose that $V^A$ has no isolated points.
Let $v\in V^A$.
We blow up the point $v$ to obtain $\widetilde{V}$ and note that $A$ has a faithful linear action on the exceptional divisor $E$.
By Fischer's theorem,
$\Br_{\mathrm{nr}}(k(E)^A)=0$, thus $\alpha$ restricts to $0\in \Br([E/A])$.
We
conclude by functoriality.
\end{proof}

\begin{exam} 
\label{exam:running}
We consider the action from \cite[Rem.\ 4.3]{KT-dp}, the projectivization of the regular representation of the Klein $4$-group $\fK_4$, and determine $\Br([\PP^3/\fK_4])$.
The action has fixed points, so $\delta_2$ is trivial.
We have $\rH^2(\fK_4,k^\times)\cong \Z/2\Z$ and $\rH^1(G,\Pic(\PP^3))=0$,
so
\[ \Br([\PP^3/\fK_4])\cong \Z/2\Z. \]
The generator $\alpha$ is not in $\Br_{\mathrm{nr}}(K)=0$,
$K=k(\PP^3)^{\fK_r}$,
so there exists $\nu\in \DVal_K$ with $\partial_\nu(\alpha)\ne 0$.
Since the $\fK_4$-action is free outside a subset of codimension $2$, we have to blow up $\PP^3$ to find a divisor giving such a $\nu$.
See Section \ref{sect:destack} for a systematic approach to testing for ramification.
\end{exam}

\section{Basic cases}
\label{sect:basic}
Our formalism permits a uniform treatment of several cases.

\subsection*{Linear actions}
The main result of Bogomolov \cite{Bog-linear} tells us that for a faithful linear representation $V^\circ$ of a finite group $G$, the field of invariants $K=k(V^\circ)^G$ has unramified Brauer group
\begin{equation}
\label{eqn:ind}
\Br_{\mathrm{nr}}(K)\cong \mathrm{B}_0(G).
\end{equation}

We apply our formalism to the standard equivariant compactification $V=\bP(1\oplus V^\circ)$ of $V^\circ$.
The $G$-action on $V$ has a fixed point, thus $\delta_2=0$. Moreover, $\rH^1(G,\Pic(V))=0$. It follows that
$\Br([V/G])$ is identified with
$\rH^2(G,k^\times)$, which we have already identified with
$\rH^2(G)=\rH^2(G,\Q/\Z)$.
The middle term in the chain of inclusions \eqref{eqn.containments} is
\[
\Br([V/G])\cong \rH^2(G).
\]
Here, subgroups of each side are identified by Bogomolov's result \eqref{eqn:ind}.

For the containment $\Br_{\mathrm{nr}}(K)\subset \mathrm{B}_0(G)$ we use
Fischer's theorem (Section \ref{sect:gen}).
If $\alpha\in \Br_{\mathrm{nr}}(K)$,
then
$\alpha_A\in \Br_{\mathrm{nr}}(k(V)^A)=0$
for $A\in \mathcal{B}_G$.
Thus the class in $\rH^2(G)$,
corresponding to $\alpha$, lies in $\ker(\mathrm{res}_A)$.

For the reverse containment we use the
equality \eqref{eqn.bicyclicnr}, recalled in the Introduction.
Suppose $\alpha\in \Br([V/G])$ corresponds to a class in $\mathrm{B}_0(G)$.
Then $\alpha_A=0$ for $A\in \mathcal{B}_G$.
So $\alpha_A\in \Br_{\mathrm{nr}}(k(V)^A)$, thus
$\alpha\in \Br_{\mathrm{nr}}(K)$.

\subsection*{Projectively linear actions}
Now we consider an action of $G$ on a projective space $V=\bP(V^\circ)$.
This arises from a representation $V^\circ$ of a cyclic extension $\widetilde{G}$ of $G$. As for linear actions we have $\rH^1(G,\Pic(V))=0$.
From Example \ref{exa.linear} we have $\gamma\in \rH^2(G)$, with 
$\mathrm{Am}(V,G)=\langle \gamma\rangle$.
We have
\[
\Br([V/G])\cong \rH^2(G)/\langle \gamma \rangle.
\]

\begin{theo}
\label{thm.proj}
For a faithful action of a finite group $G$ on a projective space $V$, corresponding to a faithful linear representation of a central cyclic extension $\widetilde{G}$ of $G$ with associated class $\gamma\in \rH^2(G)$,
we have
\[
\Br_{\mathrm{nr}}(k(V)^G)\cong\ker\Big(\rH^2(G)/\langle\gamma\rangle
\to \bigoplus_{A\in \mathcal{B}_G}\rH^2(A)/\langle \mathrm{res}^2_A(\gamma)\rangle\Big).
\]
\end{theo}

\begin{proof}
For the forwards containment,
let $A\in \mathcal{B}_G$.
We form the extension $\widetilde{A}$ of $A$ by restricting the extension $\widetilde{G}$ of $G$ and obtain
$\mathrm{B}_0(\widetilde{A})=0$ from
Lemma \ref{lem.centralofbicyclic}.
Bogomolov's result yields
\[ \Br_{\mathrm{nr}}(k(V^\circ)^{\widetilde{A}})=0, \]
and this gives us what we need, since
(with $\ell=|Z|$ in the extension \eqref{eqn.ZGextn})
\[ \Br_{\mathrm{nr}}(k(V)^A)
\cong
\Br_{\mathrm{nr}}(k(\cO_V(-\ell))^A)
\cong
\Br_{\mathrm{nr}}(k(\cO_V(-1))^{\widetilde{A}})
\cong
\Br_{\mathrm{nr}}(k(V^\circ)^{\widetilde{A}}) \]
by the stable birational invariance of the unramified Brauer group and the No-Name Lemma (see Section \ref{sect:gen}).
The reverse containment is proved as for linear actions.
\end{proof}

\subsection*{Toric actions}
Finally, we consider the $G$-action on the torus $T=\G_m^d$ given by an injective homomorphism
$$
G\hookrightarrow \GL_d(\bZ)=\GL(M),
$$
where $M=\mathfrak{X}^*(T)$ is the character lattice, and
$K=k(T)^G$.

As equivariant compactification we
take $V$ to be a smooth projective toric variety,
given by the combinatorial data of a $G$-invariant smooth projective fan of cones in $N\otimes_\Z\R$, where $N=\mathfrak{X}_*(T)$ is the cocharacter lattice.
(This exists in general; see \cite{CTHS}.)

We use a variant of \eqref{eqn.containments}, involving $\Br([T/G])$:
\[
\Br_{\mathrm{nr}}(K)\subset \ker\big(\Br([T/G])\to \Br(T)\big)\subset \rH^2(G,k(T)^\times).
\]
The middle group is accessible by the basic exact sequence of Section \ref{sect:gen}, applied to $T$.
Using the splitting given by the fixed point $1_T$ and the vanishing of $\Pic(T)$, we obtain
\[ \ker\big(\Br([T/G])\to \Br(T)\big)\cong \rH^2(G,\Q/\Z\oplus M). \]

According to Saltman \cite[Thm.\ 12]{saltmanbrauer},
the unramified Brauer group is
\[ 
\Br_{\mathrm{nr}}(K)
\cong \ker\big(\rH^2(G,\Q/\Z\oplus M)\to \bigoplus_{A\in \mathcal{B}_G}\rH^2(A,\Q/\Z\oplus M)\big).
\]
As in the other cases, the forwards containment is implied by the vanishing of $\Br_{\mathrm{nr}}(k(T)^A)$ for $A\in \mathcal{B}_G$, and the reverse containment holds by \eqref{eqn.bicyclicnr}.
So Saltman's result follows from the vanishing of $\Br_{\mathrm{nr}}(k(T)^A)$ for $A\in \mathcal{B}_G$, which we explain now, following Barge \cite{Barge}.

There is a $G$-module $M'$ with $M\oplus M'$ of finite index in a permutation module $P$
(e.g., span of boundary divisors of $V$ in the exact sequence $0\to M\to P \to \Pic(V)\to 0$,
with $M'\subset P$ giving an isomorphism
$M'\otimes \Q\to\Pic(V)\otimes \Q$).
With associated tori $T_P=\Spec(k[P])$, etc., we have $T_M=T$
and epimorphism $T_P\to T\times T_{M'}$
with finite kernel $F\subset T_P$.
On $T_P$ the translation action of $F$ and permutation action of $G$ are linear and, together, yield a semidirect product $F\rtimes G$.
For $A\in \mathcal{B}_G$ we have
\begin{equation}
\label{eqn.FAvanish}
\Br_{\mathrm{nr}}(k(T\times T_M)^A)\cong \Br_{\mathrm{nr}}(k(T_P)^{F\rtimes A})=0
\end{equation}
by Lemma \ref{lem.AsemiB} and Bogomolov's result.
The projection $T\times T_{M'}\to T$ has equivariant section $T\times \{1_{T_{M'}}\}$.
Thus the induced map
\[ \Br([T/A])\to \Br([T\times T_{M'}/A]) \]
is injective, and we obtain the desired vanishing from \eqref{eqn.FAvanish}.

\section{Grassmannians}
\label{sect:grass}
We fix notation
\[ V=\mathrm{Gr}(r,n)=\mathrm{Gr}(r,U^\circ) \]
for the Grassmannian variety of $r$-dimensional subspaces of a given $n$-dimensional $k$-vector space $U^\circ$.
Here, $1\le r\le n-1$.
Since $\Pic(V)\cong \Z$ any action yields $\rH^1(G,\Pic(V))=0$,
and
\[ \Br([V/G])\cong \rH^2(G)/\mathrm{Am}(V,G). \]

\subsection*{Automorphisms}
When $r=1$, we have projective space $U=\PP(U^\circ)$, with automorphism group $\mathrm{PGL}(U^\circ)$.
Suppose $r\ge 2$.
It is known classically
\cite{chowhomogeneous} that when
$n\ne 2r$ the automorphism group of $V$ is the same as that of $U$, i.e., $\Aut(V)= \PGL(U^\circ)$, while
for $n=2r$ there is the identity component $\PGL(U^\circ)$ of $\Aut(V)$ and a second component of automorphisms, given by isomorphisms $U^\circ\to U^{\circ\vee}$.

\subsection*{Amitsur invariant}
We recall the Amitsur invariant of a projectively linear action (Section \ref{sect:gen}).
Let $G\to \mathrm{PGL}(U^{\circ\vee})$ define a right action of $G$ on $U$, with extension \eqref{eqn.ZGextn} and compatible
\[ \widetilde{G}\to \mathrm{GL}(U^{\circ\vee}). \]
We obtain
$\gamma\in \rH^2(G)$,
with $\mathrm{Am}(U,G)=\langle\gamma\rangle$.

\begin{lemm}
\label{lem.AmitsurGrn}
Let a homomorphism $G\to \mathrm{PGL}(U^{\circ\vee})$ determine $G$-actions on $U$ and on $V$.
If the action on $U$ gives rise to $\gamma\in \rH^2(G)$, with
$\mathrm{Am}(U,G)=\langle \gamma\rangle$,
then for the action on $V$ we have $\mathrm{Am}(V,G)=\langle r\gamma\rangle$.
\end{lemm}

\begin{proof}
We consider an extension \eqref{eqn.ZGextn} with sufficiently divisible $\ell=|Z|$.
Applying the $r$th exterior power yields
the extension
\[ 1\to Z/\mu_r\to \widetilde{G}/\mu_r\to G\to 1, \]
thus $\mathrm{Am}(\PP(\bigwedge^rU^\circ),G)=\langle r\gamma\rangle$.
We conclude by applying Lemma \ref{lem.PicVW} to the Pl\"ucker embedding $V\to \PP(\bigwedge^rU^\circ)$.
\end{proof}

\begin{lemm}
\label{lem.UtimesV}
Let the notation be as in Lemma \ref{lem.AmitsurGrn}.
Then
\[
\Br_{\mathrm{nr}}(k(U)^G)\cong
\Br_{\mathrm{nr}}(k(U\times V)^G).
\]
\end{lemm}

\begin{proof}
By Example \ref{exa.linear} we have a canonical $G$-linearization of the vector bundle $\underline{U}^\circ\otimes \cO_U(1)$ on $U$, hence also of the sum of $r$ copies
$\underline{U}^{\circ{\oplus r}}\otimes \cO_U(1)$.
A similar argument supplies a canonical linearization of the tautological bundle $S$ on $V$, pulled back by the projection $\mathrm{pr}_2\colon U\times V\to V$ and tensored with $\mathrm{pr}_1^*\cO_U(1)$,
hence as well of
$\mathrm{pr}_2^*S^{\oplus r}\otimes \mathrm{pr}_1^*\cO_U(1)$.
We have a $G$-equivariant birational equivalence
\[
\underline{U}^{\circ{\oplus r}}\otimes \cO_U(1)\sim_G \mathrm{pr}_2^*S^{\oplus r}\otimes \mathrm{pr}_1^*\cO_U(1)
\]
and conclude by the stable birational invariance of the unramified Brauer group and the No-Name Lemma.
\end{proof}

\begin{lemm}
\label{lem.abelianonGr}
Let the notation be as in Lemma \ref{lem.AmitsurGrn} and $A$ an abelian subgroup of $G$ of index $d$.
We suppose that $d$ divides $r$, the order of
$\gamma$ is $d$, and
$\gamma\in \ker(\mathrm{res}_A^2)$.
Then $V^G\ne\emptyset$.
\end{lemm}

\begin{proof}
We prove the result by induction on $r$.
For the base case $r=d$,
since $\mathrm{res}_A^2(\gamma)=0$ there is a lift $A\to \mathrm{GL}(U^{\circ\vee})$ of the restriction to $A$ of the homomorphism $G\to \mathrm{PGL}(U^{\circ\vee})$.
Therefore $U^A\ne \emptyset$.
We take $z\in U^A$.
Then the linear span $\Sigma\subset U^\circ$ of the $G$-orbit of $z$ is
$G$-invariant.
Lemma \ref{lem.AmitsurGrn} implies $\dim(\Sigma)=d$, so
$[\Sigma]\in V^G$.

If $r>d$, then we take $\Sigma\subset U^\circ$ as above, $\dim(\Sigma)=d$, and let the condition to contain $\Sigma$ define a Schubert variety in $V$, isomorphic to $\mathrm{Gr}(r-d,n-d)$.
The induction hypothesis is applicable and
yields a fixed point.
\end{proof}

\subsection*{Case of projectively linear automorphisms}
Let $G$ act on $V$ via a homomorphism $G\to \PGL(U^{\circ\vee})$.
By Lemma \ref{lem.AmitsurGrn}, we have
\[ \Br([V/G])\cong \rH^2(G)/\langle r\gamma\rangle. \]

\begin{theo}
\label{thm.Grn}
Let a faithful linear action of a finite group $G$ on a projective space $U=\PP(U^\circ)$ be given, with associated class $\gamma\in \rH^2(G)$.
Then, for the induced action of $G$ on the Grassmannian $V=\mathrm{Gr}(r,U^\circ)$, we have
\[
\Br_{\mathrm{nr}}(k(V)^G)\cong\ker\Big(\rH^2(G)/\langle r\gamma\rangle
\to \bigoplus_{A\in \mathcal{B}_G}\rH^2(A)/\langle \mathrm{res}^2_A(r \gamma)\rangle\Big).
\]
\end{theo}

\begin{proof}
As in other cases, we divide the assertion into a forwards containment and a reverse containment.
The forwards containment follows from the claim, that
for $A\in \mathcal{B}_G$ we have
$\Br_{\mathrm{nr}}(k(V)^A)=0$.
The reverse containment holds by \eqref{eqn.bicyclicnr}.

We establish the claim.
Let $A\in \mathcal{B}_G$ and $\alpha\in \Br([V/A])$.
If $\alpha$ lies in $\Br_{\mathrm{nr}}(k(V)^A)$,
then
the image of $\alpha$ in $\Br([U\times V/A])$ lies in
$\Br_{\mathrm{nr}}(k(U\times V)^A)$,
which by Lemma \ref{lem.UtimesV} is isomorphic to $\Br_{\mathrm{nr}}(k(U)^A)$.
So by Theorem \ref{thm.proj},
\begin{equation}
\label{eqn.resquotient}
\alpha\in \langle \mathrm{res}_A^2(\gamma)\rangle/\langle\mathrm{res}_A^2(r\gamma)\rangle.
\end{equation}

We write $A\cong \Z/e\Z\oplus \Z/f\Z$ with $e\mid f$ and
let $d$ denote the order of the quotient group in \eqref{eqn.resquotient}.
So, $d=\gcd(r,s)$, where $s$ is the order of $\mathrm{res}_A^2(\gamma)$ in $\rH^2(A)\cong \Z/e\Z$.
We consider the subgroups $A''\subseteq A'\subseteq A$, corresponding to
\[ \Z/{\textstyle\frac{e}{s}\Z}\oplus \Z/f\Z\subseteq {\textstyle \Z/\frac{de}{s}\Z}\oplus \Z/f\Z\subseteq \Z/e\Z\oplus \Z/f\Z. \]
We have $r\gamma\in \ker(\mathrm{res}_{A'}^2)$, with the quotient group in \eqref{eqn.resquotient} mapping isomorphically to $\langle \mathrm{res}_{A'}^2(\gamma)\rangle$.
As well, $\gamma\in \ker(\mathrm{res}_{A''}^2)$.
Lemma \ref{lem.abelianonGr} is applicable and gives $V^{A'}\ne \emptyset$.
We apply Lemma \ref{lem.fixedpoint} to conclude $\alpha=0$.
\end{proof}

\subsection*{General case}
Theorem \ref{thm.Grn} gives a complete treatment of faithful actions on Grassmannians, except when $r\ge 2$ and $n=2r$, which we suppose from now on.
With the classical terminology \cite{chowhomogeneous}, $\Aut(V)$ consists of \emph{collineations}, given by projective linear automorphisms of $U^\circ$, and \emph{correlations}, given by projective isomorphisms $U^\circ\to U^{\circ\vee}$.
In formulas, for $\psi\in \GL(U^\circ)$ the collineation $L_{[\psi]}$ of $[\psi]\in \PGL(U^\circ)$ is
\[ L_{[\psi]}([\Sigma])=[\psi(\Sigma)], \]
while the correlation $C_{[\varphi]}$, for an isomorphism
$\varphi\colon U^\circ\to U^{\circ\vee}$, is
\[ C_{[\varphi]}([\Sigma])=[\Sigma']\qquad\text{with}\qquad \varphi(\sigma)(\sigma')=0 \quad  \,\forall\,\sigma\in \Sigma,\,\sigma'\in \Sigma'. \]
We have
\begin{equation}
\label{eqn.CCL}
C_{[\varphi]}\circ C_{[\varphi]}=L_{[\varphi^{{-1}\vee}\circ \varphi]}.
\end{equation}
As well, $C_{[\varphi]}$ and $L_{[\psi]}$ commute if and only if
\begin{equation}
\label{eqn.commute}
[\psi^\vee\circ \varphi\circ \psi]=[\varphi].
\end{equation}

\begin{theo}
\label{thm.Grngeneral}
Let a faithful action of a finite group $G$ on a Grassmannian $V=\mathrm{Gr}(r,n)=\mathrm{Gr}(r,U^\circ)$ be given, $\dim(U^\circ)=n$, and let $\beta\in \rH^2(G)$ be the class associated with the
projective linear action on Pl\"ucker coordinates $G\to \PGL(\bigwedge^r U^{\circ\vee})$.
Then we have
\[
\Br_{\mathrm{nr}}(k(V)^G)\cong
\ker\Big(\rH^2(G)/\langle \beta\rangle
\to \bigoplus_{A\in \mathcal{B}_G}\rH^2(A)/\langle \mathrm{res}^2_A(\beta)\rangle\Big).
\]
\end{theo}

\begin{proof}
We have $\Am(V,G)=\langle\beta\rangle$ by Lemma \ref{lem.PicVW}, applied to the Pl\"ucker embedding.
The statement is thus just Theorem \ref{thm.Grn}, unless $r\ge 2$ and $n=2r$, and the action of $G$ involves correlations;
we suppose this from now on.
We need to show that for $A\in \mathcal{B}_G$ we have $\Br_{\mathrm{nr}}(k(V)^A)=0$.
This is already known (proof of Theorem \ref{thm.Grn}) unless the action of $A$ involves correlations; we suppose this as well.
For the index $2$ subgroup $A'$ of $A$, where the action is by collineations, we have $\Br_{\mathrm{nr}}(k(V)^{A'})=0$.

Let $\alpha\in \Br([V/A])\cong \rH^2(A)/\langle\mathrm{res}^2_A(\beta)\rangle$.
If $\alpha\in \Br_{\mathrm{nr}}(k(V)^A)$, then $\alpha$ lies in the kernel of $\Br([V/A])\to \Br([V/A'])$.
The nontriviality of this kernel forces the cyclic group $\rH^2(A)$ to be of even order and the order of $\beta_A$ to be odd.
Then we conclude by Lemma \ref{lem.fixedpoint}, using the following lemma for the existence of a fixed point.
\end{proof}

\begin{lemm}
\label{lem.correlation}
Let $A$ be a bicyclic group, acting on $V=\mathrm{Gr}(r,n)$, $n=2r$.
We suppose that
if $r\ge 2$ then the action involves correlations.
We let $\beta\in \rH^2(A)$ be the class, associated with the projective linear action on Pl\"ucker coordinates.
Then $\beta$ is $2$-torsion, and we have
\[ \beta=0\qquad\text{if and only if}\qquad V^A\ne \emptyset. \]
\end{lemm}

\begin{proof}
If $r=1$ then the assertions are clear, so we suppose $r\ge 2$.
We may write
\[ A\cong \Z/e\Z\oplus \Z/f\Z, \]
where the respective generators are a correlation
$C_{[\varphi]}$ and a collineation $L_{[\psi]}$.
They commute.
In fact, the corresponding equation \eqref{eqn.commute}
may be strengthened to
\begin{equation}
\label{eqn.commute2}
\psi^\vee\circ\varphi\circ\psi=\varphi
\end{equation}
by suitably rescaling $\psi$.
From \eqref{eqn.commute2} and its equivalent form
\begin{equation}
\label{eqn.commute3}
\psi^\vee\circ\varphi^\vee\circ\psi=\varphi^\vee
\end{equation}
we obtain
\begin{equation}
\label{eqn.commute4}
\psi\circ \varphi^{{-1}\vee}\circ \varphi=\varphi^{{-1}\vee}\circ \varphi\circ \psi. \end{equation}
By \eqref{eqn.CCL} and \eqref{eqn.commute4}, the action of $A'$ (by collineations) lifts to a linear action.
So $\beta$ lies in the kernel of $\rH^2(A)\to \rH^2(A')$ and thus is $2$-torsion.

Existence of a fixed point clearly implies that $\beta$ vanishes.
It remains to show that the vanishing of $\beta$ implies the existence of a fixed point.
We do this by induction on $r$, where the base case $r=1$ is already clear.

We consider
\[ \varphi_+=\frac{1}{2}(\varphi+\varphi^\vee)\qquad\text{and}\qquad \varphi_-=\frac{1}{2}(\varphi-\varphi^\vee), \]
which determine a symmetric, respectively skew-symmetric bilinear form on $U^\circ$.
By \eqref{eqn.commute2}--\eqref{eqn.commute3} the
analogous identities for $\varphi_+$ and $\varphi_-$ also hold.
In particular, $\psi$ induces an
automorphism of $\ker(\varphi_+)$.

If $\varphi_+$ is degenerate, i.e., $\ker(\varphi_+)\ne 0$, then we may take $v\in \ker(\varphi_+)$ to be an
eigenvector of $\psi$.
There is a Schubert variety in $V$,
of $r$-dimensional spaces containing and orthogonal to $v$ (with respect to $\varphi_-$).
We apply the induction hypothesis and obtain a fixed point.

It remains to treat the case that $\varphi_+$ is nondegenerate.
Choosing an orthonormal basis of $U^\circ$ for the associated symmetric bilinear form, with dual basis of $U^{\circ\vee}$, we get a representing matrix
\[ B=I+B_- \]
for $\varphi$, where $I$ denotes the identity matrix, and the matrix $B_-$ represents $\varphi_-$ and is skew-symmetric.
The representing matrix for $\varphi^{{-1}\vee}\circ \varphi$ is
\[ C=(B^{-1})^tB. \]
We let $D$ denote the representing matrix for $\psi$; then
\[ D^tBD=B\qquad\text{and}\qquad DC=CD. \]

Suppose $B_-\ne 0$.
An orthogonal change of basis can be made to bring the matrix $B_-$ into a
normal form \cite[\S XI.4]{gantmacher}.
In the simplest case this is a block diagonal matrix with $2\times 2$-blocks
\begin{equation}
\label{eqn.block}
\begin{pmatrix} 0 & \lambda \\ -\lambda & 0 \end{pmatrix},\qquad \lambda\in k^\times,
\end{equation}
and possibly an additional zero block.
Generally there can be larger blocks, skew-symmetric analogues of the larger Jordan blocks.
But these, if present, would obstruct the diagonalizability of $C$.
Since some power of $C$ is identity, $C$ is diagonalizable, and
the normal form of $B_-$ has all nonzero blocks of the form \eqref{eqn.block}.
The fact that $D$ commutes with $C$ implies that $D$ preserves the eigenspaces of $C$.
(Always $\lambda^2\ne -1$, since $B$ is invertible, and eigenvalues $1\pm \lambda\sqrt{-1}$ of $B$ correspond to eigenvalues $(1\pm \lambda\sqrt{-1})/(1\mp \lambda\sqrt{-1})$ of $C$.)
We conclude by choosing an eigenvector and appealing to the induction hypothesis, as in the previous case.

We are left with the case $B_-=0$.
Then $B=I$, and the matrix $D$ is orthogonal.
The fixed locus
\[ V^{C_{[\varphi]}}=\{[\Sigma]\in V\,|\,\Sigma^\perp=\Sigma\} \]
is a disjoint union of two copies of the maximal orthogonal Grassmannian $\mathrm{SO}_n/\mathrm{P}_r$ (parabolic subgroup $\mathrm{P}_r$ corresponding to an end root of the Dynkin diagram $\mathsf D_r$), acted upon transitively by the orthogonal group.
We fix $[\Sigma]\in V^{C_{[\varphi]}}$ and a lift $\rho\in \GL(\bigwedge^r U^\circ)$ of $C_{[\varphi]}$.
A nontrivial homomorphism from the orthogonal group to $\{\pm 1\}$ is defined by
$\omega\mapsto \lambda'/\lambda$, where $\lambda$ and $\lambda'$ are the respective eigenvalues of $\bigwedge^r\Sigma$ and $\bigwedge^r\omega(\Sigma)$:
\[ \rho(v)=\lambda v,\qquad \rho({\textstyle\bigwedge^r}\!\omega(v))=\lambda'{\textstyle\bigwedge^r}\!\omega(v)\qquad\text{for $\textstyle v\in \bigwedge^r\Sigma$}. \]
This has to be the determinant.
So, $\beta=0$ implies $\det(D)=1$.
Then $L_{[\psi]}$ maps each component of $V^{C_{[\varphi]}}$ to itself, and
$V^A\ne \emptyset$.
\end{proof}

\section{Flag varieties}
\label{sect:flag}
We fix a $k$-vector space $U^\circ$ of dimension $n$, a positive integer $m$, and positive integers $r_1$, $\dots$, $r_m$ with
\[ 1\le r_1<\dots<r_m\le n-1. \]
In this section we extend our treatment to the partial flag variety
$$
V=\mathrm{F}\ell(r_1,\dots,r_m;n)=\mathrm{F}\ell(r_1,\dots,r_m;U^\circ)
$$
of nested subspaces of dimensions $r_1$, $\dots$, $r_m$ of $U^\circ$.
When $m=1$ this is just a Grassmannian variety (Section \ref{sect:grass}), so we assume $m\ge 2$.

\subsection*{Automorphisms}
We obtain a complete description of $\Aut(V)$ from
\cite{Demazure}.
There is an identity component $\PGL(U^\circ)$,
which is the full automorphism group except when the integers $r_1$, $\dots$, $r_m$ satisfy the symmetry condition
\[ r_i+r_{m+1-i}=n,\qquad \forall\, i. \]
In that case, as in Section \ref{sect:grass}, $\Aut(V)$ has a second component, consisting of correlations.
The action on
\[ \Pic(V)\cong \Z^m \]
is trivial (when $\Aut(V)=\PGL(U^\circ)$) or by an involutive permutation (when the symmetry condition holds).
So,
\[ \rH^1(G,\Pic(V))=0, \]
and
\[ \Br([V/G])\cong \rH^2(G)/\Am(V,G). \]

\subsection*{Projectively linear action}
Suppose that $G$ acts on $V$ via a homomorphism $G\to \PGL(U^{\circ\vee})$.
Let $\gamma\in \rH^2(G)$ be the associated class (Example \ref{exa.linear}).
Applying Lemma \ref{lem.PicVW} to the natural morphism from $V$ to the product of the Grassmannians $\mathrm{Gr}(r_i,U^\circ)$, we obtain
\[ \Am(V,G)=\langle r_1\gamma,\dots,r_m\gamma\rangle=\langle q\gamma\rangle,\qquad q=\gcd(r_1,\dots,r_m). \]

\begin{theo}
\label{thm.Fl}
Let a faithful linear action of a finite group $G$ on a projective space $U=\PP(U^\circ)$ be given, with associated class $\gamma\in \rH^2(G)$.
Then, for the induced action of $G$ on the flag variety $V=\mathrm{F}\ell(r_1,\dots,r_m;U^\circ)$ we have
\[
\Br_{\mathrm{nr}}(k(V)^G)\cong\ker\Big(\rH^2(G)/\langle q\gamma\rangle
\to \bigoplus_{A\in \mathcal{B}_G}\rH^2(A)/\langle \mathrm{res}^2_A(q \gamma)\rangle\Big),
\]
where $q=\gcd(r_1,\dots,r_m)$.
\end{theo}

The proof is similar to the case of Grassmannians (Theorem \ref{thm.Grn}).
We collect the analogous preliminary results.

\begin{lemm}
\label{lem.UtimesFl}
Let the notation be as in Theorem \ref{thm.Fl}.
Then
\[ \Br_{\mathrm{nr}}(k(U)^G)\cong \Br_{\mathrm{nr}}(k(U\times V)^G). \]
\end{lemm}

\begin{proof}
The argument is similar to the case of a Grassmannian (Lemma \ref{lem.UtimesV}), but on $V$ we have $m$ nested tautological bundles
\[ S_1\subset\dots\subset S_m \]
of ranks $r_1<\dots<r_m$.
We have an equivariant birational equivalence
\[
\underline{U}^{\circ{\oplus r_m}}\otimes \cO_U(1)\sim_G \mathrm{pr}_2^*(S_1^{\oplus r_1}\oplus S_2^{\oplus r_2-r_1}\oplus\dots\oplus S_m^{\oplus r_m-r_{m-1}})\otimes \mathrm{pr}_1^*\cO_U(1)
\]
of $G$-linearized bundles and conclude as before.
\end{proof}

\begin{lemm}
\label{lem.abelianonFl}
Let the notation be as in Theorem \ref{thm.Fl} and $A$ an abelian subgroup of $G$ of index $d$.
We suppose that $d$ divides $q$, the order of
$\gamma$ is $d$, and
$\gamma\in \ker(\mathrm{res}_A^2)$.
Then $V^G\ne\emptyset$.
\end{lemm}

\begin{proof}
We prove the result by induction on $r_m$.
By Lemma \ref{lem.abelianonGr} there exists $[\Sigma]\in \mathrm{Gr}(r_1,U^\circ)^G$.
We conclude by applying the induction hypothesis to the Schubert variety of $\Sigma_1\subset\dots\subset \Sigma_m$ with $\Sigma_1=\Sigma$.
\end{proof}

\begin{proof}[Proof of Theorem \ref{thm.Fl}]
The argument is just as in the proof of Theorem \ref{thm.Grn}.
To establish the claim, that $\Br_{\mathrm{nr}}(k(V)^A)=0$ for $A\in \mathcal{B}_G$,
we consider $\mathrm{res}^2_A(\gamma)$, whose order we denote by $s$, so the quotient group
$\langle \mathrm{res}^2_A(\gamma)\rangle/\langle\mathrm{res}^2_A(q\gamma)\rangle$ has order $d=\gcd(q,s)$;
we only need to consider elements of this quotient group, by Lemma \ref{lem.UtimesFl}.
Subgroups $A''\subseteq A'\subseteq A$ are defined just as before, and we conclude with Lemmas \ref{lem.abelianonFl} and \ref{lem.fixedpoint}.
\end{proof}

\begin{rema}
\label{rem.stablylinear}
Here, and also in the case of Grassmannians (Section \ref{sect:grass}), in case of a projectively linear action with $\gamma=0$, i.e., coming from a linear action, the action of $G$ on $V$ is stably linearizable.
We apply the construction of the proof of Lemma \ref{lem.UtimesFl}, respectively Lemma \ref{lem.UtimesV}, just without the factor $U$ and twist by $\cO_U(1)$.
\end{rema}

\subsection*{Action involving correlations}
Suppose $r_1$, $\dots$, $r_m$ satisfy the symmetry condition and the action of $G$ on $V$ involves correlations.
An index $2$ subgroup $G'$ acts by collineations with an associated class $\gamma\in \rH^2(G')$.

Let $q=\gcd(r_1,\dots,r_{[m/2]})$.
If $m$ is odd, then $n=2r_{(m+1)/2}$, and as in Section \ref{sect:grass} we have $\beta\in \rH^2(G)$, associated with the projective linear action on Pl\"ucker coordinates $G\to \PGL(\bigwedge^{r_{(m+1)/2}}U^{\circ\vee})$.
We have
\[
\Am(V,G)=\begin{cases}
\langle \mathrm{cores}_{G'}^2(q \gamma)\rangle, & \text{if $m$ is even}, \\
\langle \beta,\mathrm{cores}_{G'}^2(q \gamma)\rangle, & \text{if $m$ is odd},
\end{cases}
\]
where $\mathrm{cores}^2_{G'}\colon \rH^2(G')\to \rH^2(G)$ is the corestriction map.
This comes by applying Lemma \ref{lem.PicVW} to the product of Grassmannians $\mathrm{Gr}(r_i,U^\circ)$.
For $i=1$, $\dots$, $[m/2]$ the projective representation associated with the $G$-action on $\mathrm{Gr}(r_i,U^\circ)\times \mathrm{Gr}(r_{m+1-i},U^\circ)$ is obtained from $G'\to \PGL(U^{\circ\vee})$ by two operations.
The first, $\bigwedge^{r_i}$, multiplies the associated class by $r_i$.
The second, leading to the corestriction, is tensor induction
\cite[\S 2B]{berger}.

\begin{theo}
\label{thm.Flcorrelation}
Let a faithful action of a finite group $G$ on a flag variety $V=\mathrm{F}\ell(r_1,\dots,r_m;U^\circ)$ be given, with $m\ge 2$.
Suppose that the action of $G$ involves correlations, with index $2$ subgroup $G'$ acting by collineations leading to $\gamma\in \rH^2(G')$.
Let $\beta$ be the class associated with the projective linear action on Pl\"ucker coordinates $G\to \PGL(\bigwedge^{r_{(m+1)/2}} U^{\circ\vee})$ when $m$ is odd, $0$ when $m$ is even.
Set $q=\gcd(r_1,\dots,r_{[m/2]})$.
Then
\begin{align*}
\Br_{\mathrm{nr}}(k(V)^G&)\cong \ker
\Big(\rH^2(G)/\langle \beta,\mathrm{cores}_{G'}^2(q \gamma)\rangle \\
&\qquad\qquad\to \bigoplus_{A\in \mathcal{B}_G}\rH^2(A)/\langle \mathrm{res}^2_A(\beta),\mathrm{res}^2_A(\mathrm{cores}_{G'}^2(q \gamma))\rangle\Big).
\end{align*}
\end{theo}

\begin{proof}
We argue as in the proof of Theorem \ref{thm.Grngeneral}.
For $A\in \mathcal{B}_G$, we show $\Br_{\mathrm{nr}}(k(V)^A)=0$.
This is known (proof of Theorem \ref{thm.Fl}) when $A\subset G'$, so we suppose this is not the case.
Following the proof of Lemma \ref{lem.correlation}, we have the index $2$ subgroup $A'=A\cap G'$, whose action lifts to a linear action.
We are done, provided we can show $\mathrm{res}^2_A(\beta)=0$ implies $V^A\ne \emptyset$.

We suppose $\mathrm{res}^2_A(\beta)=0$.
Since $A'$ acts linearly, it suffices to show that $\mathrm{Gr}(r_{(m+1)/2},U^\circ)^A\ne \emptyset$ when $m$ is odd, respectively $\mathrm{Gr}(r_{m/2},U^\circ)^{A'}$ contains
a point $[\Sigma]$ sent by the correlations in $A$ to a point
\[ [\Sigma']\in \mathrm{Gr}(r_{\frac{m}{2}+1},U^\circ)\qquad\text{with}\qquad\Sigma\subset \Sigma' \]
when $m$ is even.
The argument is as in the proof of
Lemma \ref{lem.correlation}, exactly so when $m$ is odd, differing slightly in the treatment of the last case when $m$ is even.
When $B_-=0$ (notation of the proof of Lemma \ref{lem.correlation}), the fixed locus of $\mathrm{Gr}(r_{m/2},U^\circ)$ ($m$ even) for the correlation is a single copy of an orthogonal Grassmannian, thus has a fixed point.
\end{proof}

\section{General approach via destackification}
\label{sect:destack}
Let $\cX=[V/G]$ be given, where $V$ is a smooth projective rational variety and $G$ acts generically freely.
We suppose that $\Br(\cX)$ has been determined, as outlined in Section \ref{sect:stack}, in particular, an element of $\Br(\cX)$ is given by an element of $\rH^2(G,k(V)^\times)$.
Here we describe a procedure to decide whether a given element of $\Br(\cX)$ lies in $\Br_{\mathrm{nr}}(k(V)^G)$.

\subsection*{Root stacks}
Let $\cX$ be a smooth DM stack and
$\cD$ a divisor on $\cX$.
For a positive integer $r$ there is the \emph{root stack}
\[ \sqrt[r]{(\cX,\cD)} 
\]
of \cite[\S 2]{cadman}, \cite[App.\ B]{AGV}, which is again smooth, provided $\cD$ is smooth.
The root stack has the same set of $k$-points and the same coarse moduli space as $\cX$, but has stabilizer groups extended by $\mu_r$ along $\cD$.

The \emph{iterated root stack} along a simple normal crossing divisor
$\cD=\cD_1\cup\dots\cup\cD_{\ell}$ on $\cX$
\cite[Defn.\ 2.2.4]{cadman} is determined by an
$\ell$-tuple of positive integers
$\mathbf{r}=(r_1,\dots,r_{\ell})$.
This stack $\sqrt[\mathbf{r}]{(\cX,\cD)}$ is obtained by iteratively performing the $r_i$th root stack construction along each divisor $\cD_i$.

An in-depth treatment of the birational geometry of DM stacks,
including background on topics such as root stacks, is given in \cite{KT-stacks}.

\subsection*{Set-up}
To start, we replace $\cX=[V/G]$ by a smooth DM stack $\cX'$ with smooth coarse moduli space and proper birational morphism to $\cX$.

This is achieved via functorial destackification \cite{bergh}, \cite{berghrydh}.
The outcome is a sequence of stacky blow-ups whose composite $\cX'\to \cX$ is as desired.
Here, a stacky blow-up is either a usual blow-up along a smooth center or a root stack operation along a smooth divisor.
The coarse moduli space $X'$ of $\cX'$ is a smooth projective variety with a simple normal crossing divisor $D=D_1\cup\dots\cup D_\ell$ on $X'$, such that $\cX'\cong \sqrt[\mathbf{r}]{(X',D)}$ is an iterated root stack of $D$.

The morphism $\cX'\to \cX$ is not necessarily representable.
Indeed, a (nontrivial) root stack operation adds stabilizers along a divisor.
The corresponding \emph{relative} coarse moduli space is a stack $\mathsf{X}'$ with representable morphism to $\cX$.
Since $\cX$ has a representable morphism to $BG$, so does $\mathsf{X}'$, i.e.,
$\mathsf{X}'\cong [V'/G]$ for some projective variety $V'$.
The variety $V'$ is normal, but not necessarily smooth.
We have the diagram
\[ \xymatrix{\cX' \ar[r] \ar[dr] & [V'/G]\ar[d] \ar[r] & X' \\
& \cX} \]
with $2$-commutative triangle.
The vertical morphism is representable, induced by a $G$-equivariant birational proper morphism $V'\to V$.

Let $M=k(V)$.
Suppose we are given $\beta\in \rH^2(G,M^\times)$, representing $\alpha\in \Br([V/G])$.
We explain how to check whether $\alpha$ has vanishing residue along a divisor of $X'$.
It is only necessary to check this for the finitely many divisors of $X'$, where $\cX'$ has nontrivial generic stabilizer.
We have $\alpha\in \Br_{\mathrm{nr}}(M^G)$ if and only if these residues vanish.

Let $D'\subset X'$ be such a divisor, and let $D$ be a divisor in $V'$, mapping to $D'$ in $X'$.
We let $Z$ denote the stabilizer and $I$ the inertia of $D$, so $I$ is cyclic and central in $Z$.
The induced action of $\overline{Z}=Z/I$ on $D$ is faithful, and we have
$k(D)^{\overline{Z}}\cong k(D')$. Let $n=|I|$.

By the standard behavior of residue under extensions \cite[Thm.\ 10.4]{saltmanlectures}, the residue of $\alpha$ along $D'$ in $X'$ is equal to the residue of the restriction of $\alpha$ to $\Br(M^Z)$ along $D/Z$ in $V'/Z$.

We introduce notation for DVRs, fraction fields, and residue fields:
\begin{itemize}
\item $V'/Z$: 
The local ring of $V'/Z$ at the
generic point of $D/Z$ will be denoted by $R$;
fraction field $K=M^Z$, residue field $\kappa=k(D)^{\overline{Z}}$.
\item $V'/I$: 
The local ring of $V'/I$ at the generic point of $D$ will be denoted by $S$; fraction field $L=M^I$, residue field $\lambda=k(D)$.
\item $V'$:
The local ring of $V'$ at the generic point of $D$ will be denoted by $T$; 
fraction field $M$, residue field $\lambda$.
\end{itemize}
The respective maximal ideals will be denoted by $\mathfrak{m}_R$, etc.

\subsection*{Residue I}
Certainly, a necessary condition for the vanishing of the residue of $\alpha$ along $D'$ is the vanishing of the residue of the restriction of $\alpha$ to $\Br(L)$ along $D$.
We explain the computation of this residue.
The restriction of $\alpha$ is represented by
\[ \beta|_I\in \rH^2(I,M^\times)\cong L^\times/\rN_{M/L}(M^\times)=S^\times/\rN_{M/L}(T^\times). \]
Let $v\in S^\times$ be a representative of $\beta|_I$.
Then the residue of the restriction of $\alpha$ to $\Br(L)$ along $D$ is
\[ [\bar v]\in \lambda^\times/\lambda^{\times n}. \]

If $[\bar v]\ne 0$, then we have detected a nontrivial residue of $\alpha$, and we stop the computation.

\subsection*{Reduction to cocycle for $\overline{Z}$}
Continuing with the above notation, we suppose $[\bar v]=0$.
By making a suitable choice of representative $v$ we may suppose that
\[ v\in 1+\mathfrak{m}_S. \]
We let
$E\subset 1+\mathfrak{m}_S$ denote the subgroup generated by $(1+\mathfrak{m}_B)^n$ and the Galois orbit of $v$.
We define
$L'=L(E^{1/n})$ and $M'=L'M$; these are Kummer extensions of $L$.
We now show that, there is a Kummer extension $K'/K$ with $K'L=L'$ and $[K':K]=[L':L]$.

A choice of maximal ideal of the integral closure of $S$ in $L'$ determines, by localization, a DVR $S'$ with residue field $\lambda$.
The Kummer pairing of $\Gal(L'/L)$ with $E$
extends to a pairing
\[ \Gal(L'/K)\times E\to \mu_n. \]
The induced homomorphism $\Gal(L'/K)\to \Hom(E,\mu_n)\cong \Gal(L'/L)$
determines a direct product decomposition
\[ \Gal(L'/K)\cong \Gal(L'/L)\times \overline{Z} \]
and thus a Kummer extension
\[ K'=L'^{\overline{Z}} \]
of $K$ with $K'L=L'$.
The corresponding DVR $R'$ has residue field $\kappa$.

If we replace the tower of fields $M/L/K$ by $M'/L'/K'$
and pass from $\beta|_Z\in \rH^2(Z,M^\times)$ to
$\beta'\in \rH^2(Z,{M'}^\times)$, the residue does not change, and we have $v\in({L'}^\times)^n$.
So
\[ \beta'\in \ker\big(\rH^2(Z,{M'}^\times)\to \rH^2(I,{M'}^\times)\big). \]

\subsection*{Residue II}
We keep the above notation but revert to the notation $M/L/K$ for the tower of fields.
So we have reduced to the case
\[ \beta|_Z\in \ker\big(\rH^2(Z,M^\times)\to \rH^2(I,M^\times)\big). \]
Then, by the Hochschild-Serre spectral sequence and Hilbert's Theorem 90, $\beta|_Z$ is the image, under the inflation map, of some
\[ \gamma\in \rH^2(\overline{Z},L^\times). \]

Since the $\overline{Z}$-Galois extension $L/K$ is associated with a unramified extension of DVRs,
the residue is determined by the procedure described in \cite[\S III.2]{GB}.
We apply the valuation
\[ \mathrm{val}\colon L^\times\to \Z \]
to obtain
$\mathrm{val}(\gamma)\in \rH^2(\overline{Z},\Z)$.
Now the residue is the class associated with $\mathrm{val}(\gamma)$ under the isomorphism
\[ \Hom(\overline{Z},\Q/\Z)=\rH^1(\overline{Z},\Q/\Z)\cong \rH^2(\overline{Z},\Z). \]

\begin{exam}
\label{exa.destaP3modK4}
For the quotient stack $[\bP^3/\fK_4]$ of Example \ref{exam:running}, with
Brauer group of order $2$ generated by $\alpha$,
destackification is achieved by
\begin{itemize}
\item blowing up the fixed points to produce exceptional divisors $E_i$ ($i\in \{0,\dots,3\}$),
\item blowing up the proper transforms of the intersections of pairs of coordinate hyperplanes to yield exceptional divisors $E_{ij}$ ($i$, $j\in \{0,\dots,3\}$, $i<j$), and
\item blowing up the intersections of the proper transforms of the exceptional divisors from the first blow-up with the proper transforms of the coordinate hyperplanes, leading to exceptional divisors $E'_{cd}$ ($c$, $d\in \{0,\dots,3\}$, $c\ne d$).
\end{itemize}
As indicated in \cite[Rem.\ 3.3]{oesinghausconic}, since only $\Z/2\Z$ and $\fK_4$ occur as stabilizer groups, destackification is achieved with just ordinary blow-ups (no nontrivial root stack operations).
So $\cX'=[V'/\fK_4]$.
Along the divisors $E_{ij}$ and $E'_{cd}$ the generic stabilizer has order $2$.
Let $D\subset V'$, over $D'\subset X'$, be one of the divisors with nontrivial generic stabilizer.
In local coordinates $x$, $y$, $z$, we have $D$ given by $x=0$, where $\fK_4$ acts by distinct nontrivial characters on $x$ and $y$ and acts trivially on $z$.
We have $|I|=2$ and
$\beta\in \rH^2(\fK_4,k(x,y,z)^\times)$, 
given by a $\mu_2$-valued cocycle and
image under the inflation map of
$[x^2]\in \rH^2(\fK_4/I,k(x^2,y,z)^\times)$ (with the conventions of Section \ref{sect:gen} for cyclic group cohomology).
The residue is given by the nontrivial homomorphism $\fK_4/I\to \Q/\Z$.
\end{exam}

\begin{exam}
\label{exa.M06}
Consider the action of
\[ G=\mathfrak{A}_4\cong \langle(135)(246),(12)(34),(12)(56)\rangle\subset \fS_6 \]
on $V=\overline{\cM}_{0,6}$.
This is a \emph{nonstandard} $\mathfrak{A}_4$ in $\fS_6$, \emph{not} fixing a plane
in the Segre cubic model.
Actions fixing a plane, such as the Klein $4$-group $\fK_4\subset G$, are birational to actions on toric varieties, see \cite[Section 6]{CTZ}.
Restriction to the Klein $4$-group induces an isomorphism
$$
\rH^2(G)\cong \rH^2(\fK_4)\cong\bZ/2\bZ.
$$
As well, $V^G$ is nonempty, with
$$
\rH^1(G, \Pic(V))\cong \rH^1(\fK_4, \Pic(V))\cong \bZ/2\bZ.
$$
So
$$
\Br([V/G])\cong \Br([V/\fK_4])\cong \bZ/2\bZ\oplus \bZ/2\bZ.
$$
It is known that $\Br_{\mathrm{nr}}(k(V)^{\fK_4})=0$ (since the $\fK_4$-action is birational to a toric action, and the rationality of such a quotient is a special case of 
\cite[Thm. 1.2 and 1.3]{KP}); 
consequently, 
\[ \Br_{\mathrm{nr}}(k(V)^G)=0. \]
\end{exam}

\bibliographystyle{plain}
\bibliography{brauerquot}

\end{document}